\documentclass[a4paper]{article}

\usepackage{amsmath,amssymb,amsthm}
\usepackage{centernot,braket}
\usepackage{lipsum}
\usepackage{calc}
\usepackage{xcolor}
\usepackage{csquotes}
\usepackage{comment}
\usepackage{fancyhdr}
\usepackage{titling}
\usepackage[textwidth=4cm]{todonotes}
\usepackage{hyperref}
\usepackage{nicematrix}

\usepackage{tikz}
\usetikzlibrary{cd}
\def\temp{&} \catcode`&=\active \let&=\temp

\usepackage{setspace}
\hypersetup{colorlinks,citecolor=blue}
\setuptodonotes{color=blue!10,bordercolor=blue!10}

\tikzcdset{row sep/normal=3.6em}
\tikzcdset{column sep/normal=5em}
\tikzset{help lines/.style={very thin, color=lightgray, dashed}}
\tikzset{axis lines/.style={very thin, color=lightgray}}

\makeatletter
\def\mathcolor#1#{\@mathcolor{#1}}
\def\@mathcolor#1#2#3{%
  \protect\leavevmode
  \begingroup
    \color#1{#2}#3%
  \endgroup}
\makeatother

\let\originalleft\left
\let\originalright\right
\renewcommand{\left}{\mathopen{}\mathclose\bgroup\originalleft}
\renewcommand{\right}{\aftergroup\egroup\originalright}

\newtheoremstyle{MYBREAK}
  {}          
  {}          
  {\itshape}  
  {}          
  {\bfseries} 
  {:}         
  {\newline}  
  {}          
\newtheoremstyle{MYPLAIN}
  {\topsep}   
  {\topsep}   
  {\itshape}  
  {15pt}          
  {\bfseries} 
  {}         
  {5pt plus 1pt minus 1pt} 
  {}          

\theoremstyle{definition}
\newtheorem{DEF}{Definition}[section]

\theoremstyle{plain}
\newtheorem{PROP}[DEF]{Proposition}
\newtheorem{LMM}[DEF]{Lemma}
\newtheorem{THM}[DEF]{Theorem}

\theoremstyle{MYBREAK}

\newcommand{\TAB}{\quad}

\newcommand{\CSTAR}{\ensuremath{C^*}}

\DeclareMathSymbol{\mlq}{\mathord}{operators}{'134}
\DeclareMathSymbol{\mrq}{\mathord}{operators}{'42}

\newenvironment{TIKZCD}{\[\begin{tikzcd}}{\end{tikzcd}\]\ignorespacesafterend}

\renewcommand{\subset}{\subseteq}

\renewcommand{\phi}{\varphi}
\renewcommand{\epsilon}{\varepsilon}

\makeatletter
\newcommand{\xRightarrow}[2][]{\ext@arrow 0359\Rightarrowfill@{#1}{#2}}
\makeatother

\newcommand{\INTEGERS}{\mathbb{Z}}
\newcommand{\REALS}{\mathbb{R}}
\newcommand{\COMPLEX}{\mathbb{C}}

\DeclareMathOperator{\trace}{tr}

\newcommand{\REMAINDER}{\mathrm{rem}}

\newcommand{\inputs}{\mathrm{in}}
\newcommand{\outputs}{\mathrm{out}}

\newcommand{\Alice}{\operatorname{Alice}}
\newcommand{\Bob}{\operatorname{Bob}}
\newcommand{\player}{\operatorname{player}}

\newcommand{\twoplayer}{\operatorname{two~player}}

\DeclareMathOperator{\pos}{pos}
\DeclareMathOperator{\proj}{proj}

\newcommand{\UHF}{\mathrm{UHF}}
\newcommand{\TWOONEFACTOR}{$\mathrm{II}_1$-factor}
\newcommand{\hyperfiniteTWOONE}{\mathcal{R}}
\newcommand{\Romega}{\mathcal{R}^\omega}

\newcommand{\tensor}{\mathbin{\otimes}}

\title{Connes implies Tsirelson:\\A simple proof}
\author{Alexander Frei}
\date{\today}

\begin{document}

\maketitle
More precisely, we give a quick and very simple proof of \enquote{the Connes\linebreak embedding problem implies the synchronous Tsirelson conjecture} that relies on only \emph{two elementary ingredients:}
1) the well-known description of synchronous correlations as traces on the algebra per player $\CSTAR(\player)=\CSTAR(\inputs|\outputs)$\linebreak
and 2) an elementary lifting result by Kim, Paulsen and Schafhauser.

Moreover, this bypasses every of the deep results by Kirchberg as well as any other implicit reformulation as the microstates conjecture and thelike.

Meanwhile, we also give a different construction of Connes' algebra $\Romega$ appearing in the Connes embedding problem, which is more suitable for the purposes of quantum information theory and much easier to comprehend for the reader without any prior knowledge in operator algebras.

Most importantly, however, we present this proof for the following reason:
Since the recent refutation of the synchronous Tsirelson conjecture by MIP*=RE, there exists a nonlocal game which violates the synchronous Tsirelson conjecture,
and by the proof of MIP*=RE even a synchronous such game.
The approach however is based on contradiction with the undecidability of the Halting problem, and so remains implicit.
As such the quest now has started to give an \emph{explicit example of a synchronous game} violating the synchronous Tsirelson conjecture together with a direct argument for the failure, and the current article serves as a direct translation to the corresponding operator algebra and its tracial state violating the Connes embedding problem.

Meanwhile the author would like to stress that he does not take any credit for any of the results as already available.
Our only contribution lies in combining the well-known results from quantum information theory with the lifting result as established by Kim Paulsen and Schafhauser in \cite{KIM-PAULSEN-SCHAFHAUSER}.

\section[Player algebra]{Algebra per player: $\CSTAR(\inputs|\outputs)$}

Let us swiftly introduce the algebra per player, which defines one of the two relevant algebras appearing in our main result below.
It is well-known (by Gelfand and Pontryagin duality) that the universal \CSTAR-algebra generated by a projection valued measure agrees with the group \CSTAR-algebra%
\[
  \CSTAR(\INTEGERS/A)=\CSTAR(e_1,\ldots,e_A\in\proj|e_1+\ldots +e_A=1)
\]
and as such we obtain as the universal \CSTAR-algebra per player with a given number of questions and answers the amalgamated free product
\[
  \CSTAR(X=\inputs|A=\outputs):=\overbrace{\CSTAR(\INTEGERS/A)*_1\ldots*_1\CSTAR(\INTEGERS/A)}^\text{$X$-many} = \CSTAR(\INTEGERS/A*\ldots*\INTEGERS/A).
\]
From here one may (we won't be needing it though) define the algebra for two-player games as the maximal tensor product of each player's algebra (with corresponding question and answer sets)
\[
  \CSTAR(\twoplayer)=\CSTAR(\Alice)\tensor\CSTAR(\Bob)=\CSTAR(X|A)\tensor\CSTAR(Y|B)
\]
and any quantum commuting strategy arises simply as a state
\[
  \phi:\CSTAR(\twoplayer)\to\COMPLEX:\TAB p(ab|xy)=\phi\Big(e(a|x)\tensor e(b|y)\Big).
\]
For us the above algebra per player however will be sufficient since we will be dealing with synchronous strategies exclusively, and whence tracial states on the algebra per player.
Having intoduced the first algebra relevant for our main result, let us proceed to the second one.

\section[Connes algebra]{Connes' algebra: $\Romega$}

We introduce in this section the tracial ultrapower of the hyperfinite \TWOONEFACTOR, in short also refered to as Connes' algebra, which defines the main player in the Connes embedding problem.
For this we will pursue another construction for the Connes algebra, which defines a more suitable approach when working in quantum information theory
(and which may be also much easier to comprehend for the reader without any prior knowledge in operator algebras).

But before we do so, let us first describe how the usual construction of the hyperfinite \TWOONEFACTOR\ goes along, just for having both constructions available.
Pick your favorite UHF-algebra of infinite-type like
\begin{gather*}
  \UHF=M_2\tensor M_2\tensor\ldots=\bigotimes^\infty M_2,\TAB\text{or}\TAB\UHF=\left(\bigotimes^\infty M_3\right)\tensor\left(\bigotimes^\infty M_7\right),\\[2\jot]
  \TAB\text{or}\TAB \UHF=\left(\bigotimes^\infty M_2\right)\tensor\left(\bigotimes^\infty M_3\right)\tensor\ldots=\bigotimes_\text{$p$: all primes}\left(\bigotimes^\infty M_p\right).
\end{gather*}
and consider its unique tracial state, for example
\[
  \trace:\UHF=\bigotimes^\infty M_2\to\COMPLEX:\TAB \trace(a_1\tensor a_2\tensor 1\tensor\ldots)=\trace(a_1)\trace(a_2).
\]
Just as an intermediate result --- which may however also be omitted!%
\footnote{Due to \cite[lemma A.9]{BROWN-OZAWA}, which following the proof also works for \CSTAR-algebras,\\ and a Kaplansky density argument.}
--- one may embed the chosen UHF-algebra in the GNS-representation for the tracial state. Taking the von-Neumann algebraic completion of the UHF-algebra within the GNS-representation produces then the hyperfinite \TWOONEFACTOR:
\[
  \hyperfiniteTWOONE = \overline{\UHF}^\mathrm{w\star}\subset B(\overline{\UHF}^2)
\]
The story however wouldn't be over here: From here we would still need to proceed to the ultrapower of the hyperfinite \TWOONEFACTOR\ as follows:
Take the infinite repeated power of either the UHF-algebra from above (or alternatively the hyperfinite \TWOONEFACTOR\ that arose from the completion),
which we denote for shorthand by
\[
  \ell^\infty(\UHF):=\left\{a\in\prod_n^\infty\UHF\Bigg|\sup_n\|a_n\|<\infty\right\} = \prod_n^\infty\UHF.
\]
The resulting operator algebra comes together with the induced trace
\[
  \trace_\omega:\ell^\infty(\UHF)\to\COMPLEX:\TAB \trace_\omega(a) = \lim_{n\to\omega}\trace(a_n)
\]
where the limit is taken along any free ultrafilter.
In order to render the trace faithful one passes, together with the trace, to the following quotient which defines the so-called tracial ultrapower
\[
  \UHF^\omega:=\frac{\ell^\infty(\UHF)}{c_\omega(\UHF;\trace):=\{a\in\ell^\infty(\UHF)|\trace_\omega(a^*a)=0\}}.
\]
This would be the final object appearing in the Connes embedding problem:
\[
  \Romega = \frac{\ell^\infty(\hyperfiniteTWOONE)}{c_\omega(\hyperfiniteTWOONE;\trace)} = \frac{\ell^\infty(\UHF)}{c_\omega(\UHF;\trace)}=\UHF^\omega
\]
Now, while this construction has its merits in the classification of factors due to the hyperfinite \TWOONEFACTOR\ appearing in the intermediate construction, it would be way too overloaded for applications in quantum information theory, and in particular for the relation between the Connes embedding problem and the Tsirelson conjecture. So we refrain from this construction!

Instead we introduce now a different construction which, as promised above, defines a more convenient approach when working in quantum information theory.
For this we shortcut the construction above and basically start at its very end. More precisely, we take the infinite product on matrices of arbitrary sizes, which we denote suggestively as above by
\[
  \ell^\infty(M):=\left\{a\in\prod_n^\infty M_n\Bigg|\sup_n\|a_n\|<\infty\right\} = \prod_n^\infty M_n.
\]
As above, this comes together with the induced trace
(along any free ultrafilter)
\[
  \trace_\omega:\ell^\infty(M)\to\COMPLEX:\TAB \trace_\omega(a) = \lim_{n\to\omega}\trace_n(a_n)
\]
which we render faithful as above by passing to the quotient by
\[
  c_\omega(M;\trace):=\left\{a\in\prod_nM_n|\trace_\omega(a^*a)=\lim_n\trace_n(a_n^*a_n)=0\right\}.
\]
Summarizing the construction: considering sequences of matrices of arbitrary sizes allows us to pickup all finite-dimensional representations, and passing to the quotient allows us to also do so approximately.

On the other hand, the surprising feature (surprising just to some extend) is that this defines also the same tracial ultrapower as before,
\[
  M^\omega = \frac{\ell^\infty(M)}{c_\omega(M;\trace)} = \frac{\ell^\infty(\hyperfiniteTWOONE)}{c_\omega(\hyperfiniteTWOONE;\trace)} = \Romega.
\]
This follows basically by some Kaplansky density type argument and a careful diagonal reindexing --- just with many more steps (see also the footnote above).
Since this would however escape the scope of the current article, we refrain from presenting a proof in here and instead leave it to the experienced reader.

\enlargethispage{\baselineskip}
Now the left-hand side construction we have just provided is much closer to correlations, and so we will from now on use our construction in what follows.
So far on the operator algebra appearing in the Connes embedding problem.

\section[Ingredients]{Two ingredients}

In this section we prepare two simple ingredients, proposition \ref{SYNC=TRACES} and proposition \ref{LIFTING-PROBLEM} below, which are all what is needed for \enquote{the Connes embedding problem implies the synchronuous Tsirelson conjecture}.
Let us start with the first.

\begin{PROP}[{\cite[corollary 5.6]{PSSTW-2016}}]\label{SYNC=TRACES}
  The set of synchronous quantum-commuting correlations is realized by the trace space restricted on two-moments:
  \[
    \tau\in T\CSTAR(\inputs|\outputs):\TAB p(ab|xy)=\tau(e_{ax}e_{by}).
  \]
  Similarly the set of finite dimensional synchronous correlations is realized by those of such traces which live on finite dimensional quotients thereof
  \begin{TIKZCD}
    \CSTAR(\inputs|\outputs)\rar & \mathrm{some\ fin\ dim\ quotient}\rar{\mathrm{some\ trace}} & \COMPLEX
  \end{TIKZCD}
  and then restricted on two-moments as above.
\end{PROP}

We skip the proof since it is fairly elementary and well-known.
The second ingredient is the following lifting result found by Kim, Paulsen and Schafhauser,
which we recall together with its proof for convenience of the reader.

\begin{PROP}[{\cite[lemma 3.5]{KIM-PAULSEN-SCHAFHAUSER}}]\label{LIFTING-PROBLEM}
  The following lifting problem has a solution:
  Every representation into the hyperfinite $\mathrm{II}_1$-factor lifts to matrices
  \begin{TIKZCD}
    & \ell^\infty(M) \dar \\
    \CSTAR(\INTEGERS/m)=\overbrace{\COMPLEX\oplus\ldots\oplus\COMPLEX}^{\text{$m$-many}} \rar\urar[dashed] &
    \ell^\infty(M)/c(M,\trace)=\Romega
  \end{TIKZCD}
  and unital representations may be lifted unitally.
  As a consequence, the lifting problem also has a solution for any number of inputs and outputs
  \[
    \CSTAR(\underbrace{\INTEGERS/m*\ldots*\INTEGERS/m}_\text{$n$-many})=\CSTAR(n=\inputs|m=\outputs).
  \]
\end{PROP}

Before we begin with the proof itself, let us note that the lifting problem easily admits a solution by positive maps:
Any such representation is determined by some tuple of mutually orthogonal projections in the quotient
\[
  \COMPLEX\oplus\ldots\oplus\COMPLEX\to\ell^\infty(M)/c(M,\trace):\TAB e_1\mapsto q_1,\TAB \ldots\TAB e_m\mapsto q_m
\]
and so in particular by some tuple of positive elements. Any single positive element in a quotient may however be easily lifted itself
\[
  \pi:A\to B\to0:\TAB\TAB \forall b\in\pos(B)\ \exists a\in \pos(A):\TAB \pi(a)=b
\]
and so also an entire tuple of positive elements --- each element one-by-one.\\
Put together this defines a solution by some positive map
\[
  \COMPLEX\oplus\ldots\oplus\COMPLEX\to \ell^\infty(M):\TAB \TAB e_1\mapsto a_1,\TAB \ldots\TAB e_m\mapsto a_m.
\]
At the same time one may always arrange for such a lift without increasing the norm of each lift and so arrange for contractions
\[
  \|a_1\|=\|q_1\|\leq1,\TAB \ldots\TAB ,\|a_m\|=\|q_m\|\leq1.
\]
On the other hand, in case of some unital representation one may moreover arrange for a unital lift by adding the remainder on the last,
\[
  \REMAINDER = 1-(a_1+\ldots+a_m):\TAB a_m' := a_m+\REMAINDER.
\]
The gist is now to also arrange for some projection valued lift: this is one of the main accomplishments by Kim, Paulsen and Schafhauser in their article on synchronous games.
At the heart of this problem lies the following technical result, which we formulate in its slightly improved, optimal version.\\
Either such bound (be it optimal or just some upper bound) then allows us to deform the tuple of positive elements into an actual set of projections.
\begin{LMM}[{compare \cite[lemma 3.4]{KIM-PAULSEN-SCHAFHAUSER}}]\label{2-NORM-BOUND}
  Consider for any positive matrix contraction $a\in M_n$ the spectral projection onto the interval $[1/2,1]\subset\REALS$:
  \[
    p(a):=1_{[1/2,1]}(a)\in M_n:\TAB p(a)^2=p(a)=p(a)^*.
  \]
  Then it holds the upper bound for the distance
  \[
    \TAB\|a - p(a)\|_\phi\leq 2\|a^2-a\|_\phi
  \]
  uniformly in every positive matrix contraction
  and any state 2-norm
  \[
  \|x\|_\phi^2=\braket{x|x}_\phi=\phi(x^*x)=\phi(|x|^2)
  \]
  and the bound above is optimal for faithful states.
\end{LMM}
\begin{proof}
  We give a slightly different proof than in \cite{KIM-PAULSEN-SCHAFHAUSER}: Instead the optimal bound can be easily read off from the spectrum as follows.
  For this denote for shorthand the left-hand and right-hand side as
  \[
    f(x):=x-h(x),\TAB g(x):=x^2-x=x(1-x).
  \]
  Since states are positive, we have as a sufficient condition
  \[
    f(a)^2\leq g(a)^2 \implies \phi\Big(f(a)^2\Big)\leq\phi\Big(g(a)^2\Big).
  \]
  This however can be now read off from the spectrum as
  \[
    \forall x\in\sigma(a):\TAB |f(x)|\leq |g(x)|
  \]
  which in our case  boils down to the condition
  \[
    \begin{cases}
    |x|\leq 2|x|\cdot|1-x| & \text{for}\ x\in[0,1/2]\cap\sigma(a), \\
    |1-x|\leq 2|x|\cdot|1-x| & \text{for}\ x\in[1/2,1]\cap\sigma(a).
    \end{cases}
  \]
  Finally note that the bound is optimal for faithful states: simply use some matrix whos spectrum contains the eigenvalue 1/2 --- for instance half the identity.
\end{proof}
With the previous bound at hand one may now derive the desired solution to the lifing problem, which we sketch for completeness.
\begin{proof}[Sketch of proposition \ref{LIFTING-PROBLEM} (based on the construction by \cite{KIM-PAULSEN-SCHAFHAUSER}):]
  Say we have already found a lift to some tuple of positive contractions
  \[
    a_1,\ldots,a_m\in \ell^\infty(M) = \prod_nM_n.
  \]
  While each element of the tuple is a matrix sequence itself like
  \[
    a=(a_1,a_2,\ldots)\in \ell^\infty(M)
  \]
  we will keep viewing each matrix sequence as a single element. The reader new to operator-algebraic techniques may however also savely run the following procedure indexwise for each sequence in the tuple.
  Replace the first one in the tuple by the spectral projection from lemma \ref{2-NORM-BOUND}, then cut-off the resulting projection from the next one and apply the lemma again to that,
  \[
    p_1:=p(a_1),\TAB a_2':=(1-p_1)a_2(1-p_1),\TAB p_2:=p(a_2').
  \]
  Continuing this way one needs to cut-off all the previous projections, for example
  \[
    a_3'=(1-p_1-p_2)a_3(1-p_1-p_2).
  \]
  This way we guarantee their orthogonality since for each next step
  \[
    1\leq k+1\leq m:\TAB a_{k+1}'\perp p_1,\ldots,p_k\implies p(a_{k+1}')\perp p_1,\ldots,p_k.
  \]
  The upper bound in lemma \ref{2-NORM-BOUND} now guarantees that the deformation procedure remains a lift for the original tuple in the quotient: Indeed recall that the quotient is given by the ideal
  \[
    c(M,\trace)=\{a\in\ell^\infty(M)\mid\|a\|_2=0\}
  \]
  which reads when written out as a sequence $a=(a_1,\ldots)\in\prod_nM_n$:
  \[
    \|a\|_2^2=\trace_\omega(a^*a)=\lim_{n\to\omega}\trace\Big(a_n^*a_n\Big)=0.
  \]
  The original tuple in the quotient however consists of mutually orthogonal projections and so the deformation procedure remains a lift since:
  Taking the spectral projection does not alter the equivalence class since by lemma \ref{2-NORM-BOUND}
  \[
    \|p(a)-a\|_2\leq 2\|a^2-a\|_2=0
  \]
  (it helps to view them as a sequence of matrices and one as a sequence of projection matrices, which indexwise get closer and closer in trace 2-norm),\linebreak
  nor does the cutting-off procedure since this cannot be seen by any orthogonal pair in a quotient anyways:
  \[
    \pi:A\to B\to0:\TAB\TAB\pi(a)\perp\pi(a')\implies \pi(1-a)\pi(a')\pi(1-a)=\pi(a').
  \]
  This completes the sketch of the construction for proposition \ref{LIFTING-PROBLEM}.
\end{proof}

We have now successfully established our ingredients,\\
and so we may now proceed to the main result of the article:

\section{Connes \texorpdfstring{$\implies$}{$>$} Tsirelson}
With our two simple ingredients at hand, namely proposition \ref{SYNC=TRACES} and \ref{LIFTING-PROBLEM}, we may now verify that \enquote{Connes implies the synchronous Tsirelson conjecture}.
\begin{THM}\label{CONNES-TSIRELSON}
  The Connes embedding problem implies the synchronous\linebreak Tsirelson conjecture.
  More precisely, suppose the Connes embedding problem holds true, meaning
  all tracial states approximately arise as
  the unique tracial state on the hyperfinite $\mathrm{II}_1$-factor in the sense that there exists a factorization
  \begin{TIKZCD}
    A\rar[dashed] \arrow[rr, bend right, swap, "\mathrm{some\ trace}"] &
    \Romega \rar{\trace_\omega} & \COMPLEX
  \end{TIKZCD}
  Then the synchronous Tsirelson conjecture would also hold true:
  \begin{gather*}
    C^s_{qc}(n|m) \subset \overline{C_{q}(n|m)\cap C^{s}(n|m)}\,\Big(\subset C^s_{qa}(n|m)\subset C_{qc}^{s}(n|m)\Big).
  \end{gather*}
  For comparison it still holds by \cite{KIM-PAULSEN-SCHAFHAUSER}
  \[
    \overline{C_{q}(n|m)\cap C^{s}(n|m)}=C^s_{qa}(n|m)
  \]
  which requires yet a difficult result by Kirchberg on amenable traces!\\
  As consequence, the recent refutation of the synchronous version of the Tsirelson conjecture by \cite{MIP=RE} implies the failure of the Connes embedding problem.
\end{THM}

\begin{proof}
  The result now basically follows from the previous ingredients:
  \vspace*{-\topsep}
  \vspace*{-\partopsep}
  \vspace*{-\parskip}
  \begin{enumerate}
    \item[Step 1:]Consider a synchronous quantum-commuting strategy, which by proposition \ref{SYNC=TRACES} is given by some trace restricted on two-moments,
    \[
      \tau\in T\CSTAR(\inputs|\outputs):\TAB p(ab|xy)=\tau(e_{ax}e_{by}).
    \]
    \item[Step 2:] Assuming the Connes embedding problem, all traces would approximately arise as the unique trace on the hyperfinite \TWOONEFACTOR\ and so also our trace:
    \begin{TIKZCD}
      \CSTAR(\inputs|\outputs) \rar[dashed]\arrow[rr, bend right] &
      \ell^\infty(M)/c_\omega(M,\trace)=\Romega\rar{\trace_\omega} &\COMPLEX.
    \end{TIKZCD}
    \item[Step 3:] By proposition \ref{LIFTING-PROBLEM} any such representation admits a lift
    \begin{TIKZCD}
      & \ell^\infty(M) \dar \\
      \CSTAR(\inputs|\outputs) \rar\urar[dashed] &
      \ell^\infty(M)/c_\omega(M,\trace)=\Romega.
    \end{TIKZCD}
  \end{enumerate}
  Putting all together, any such resulting lift however gives a desired approximation by finite dimensional synchronous correlations:
  \begin{TIKZCD}
    & \ell^\infty(M) \dar\drar[bend left]\drar[phantom, "{\trace_\omega}"description] \\
    \CSTAR(\inputs|\outputs)\urar &
    \ell^\infty(M)/c_\omega(M,\trace)=\Romega\rar &\COMPLEX.
  \end{TIKZCD}
  More precisely, simply recall the trace as it was originally defined:
  \[
    \prod_nM_n=\ell^\infty(M)\ni a=(a_1,a_2,\ldots):\TAB\TAB \trace_\omega(a)=\lim_{n\to\omega}\trace_n(a_n).
  \]
  This completes the proof and so we conclude as desired:
  the Connes embedding problem implies the synchronous Tsirelson conjecture.
\end{proof}

\section*{Acknowledgements}

The author would like to thank his supervisor S{\o}ren Eilers for his kind support and encouragements, as well as Ryszard Nest and Mikkel Munkholm for helpful discussions on the hyperfinite \TWOONEFACTOR.
Moreover, the author acknowledges the support under the Marie–Curie Doctoral~Fellowship~No.~801199.

\bibliography{Bibliography}

\begin{thebibliography}{JNVWY21}

\bibitem[BO08]{BROWN-OZAWA}
Nathanial Brown and Narutaka Ozawa.
\newblock {\em {C*}-{A}lgebras and {F}inite-{D}imensional {A}pproximations},
  volume~88 of {\em Graduate Studies in Mathematics}.
\newblock Amercian Mathematical Society, 2008.

\bibitem[JNVWY21]{MIP=RE}
Zhengfeng Ji, Anand Natarajan, Thomas Vidick, John Wright, and Henry Yuen.
\newblock {MIP* = RE}.
\newblock {\em Communications of the ACM}, 64(11):155--181, November 2021.

\bibitem[KPS18]{KIM-PAULSEN-SCHAFHAUSER}
Se-Jin Kim, Vern Paulsen, and Christopher Schafhauser.
\newblock A synchronous game for binary constraint systems.
\newblock {\em Journal of Mathematical Physics}, 59(3):032201, 2018.

\bibitem[PSSTW16]{PSSTW-2016}
Vern Paulsen, Simone Severini, Daniel Stahlke, Ivan Todorov, and Andreas
  Winter.
\newblock Estimating quantum chromatic numbers.
\newblock {\em Journal of Functional Analysis}, 270(6):2188--2222, 2016.

\end{thebibliography}
\bibliographystyle{alpha-all}

\end{document}